\documentclass[preprint,12pt]{elsarticle}

\usepackage[T1]{fontenc}
\usepackage{lmodern}
\usepackage{amsmath,amssymb,amsthm,mathtools}
\usepackage{booktabs}
\usepackage{enumitem}
\usepackage[colorlinks=true,linkcolor=blue,citecolor=blue,urlcolor=blue]{hyperref}
\usepackage[nameinlink,capitalize,noabbrev]{cleveref}
\usepackage{microtype}

\journal{Linear Algebra and its Applications}
\biboptions{numbers,sort&compress}

\numberwithin{equation}{section}

\newtheorem{theorem}{Theorem}[section]
\newtheorem{proposition}[theorem]{Proposition}
\newtheorem{lemma}[theorem]{Lemma}
\newtheorem{corollary}[theorem]{Corollary}
\newtheorem{example}[theorem]{Example}
\theoremstyle{definition}

\theoremstyle{remark}
\newtheorem{remark}[theorem]{Remark}

\crefname{theorem}{Theorem}{Theorems}
\Crefname{theorem}{Theorem}{Theorems}
\crefname{proposition}{Proposition}{Propositions}
\Crefname{proposition}{Proposition}{Propositions}
\crefname{lemma}{Lemma}{Lemmas}
\Crefname{lemma}{Lemma}{Lemmas}
\crefname{corollary}{Corollary}{Corollaries}
\Crefname{corollary}{Corollary}{Corollaries}
\crefname{example}{Example}{Examples}
\Crefname{example}{Example}{Examples}
\crefname{remark}{Remark}{Remarks}
\Crefname{remark}{Remark}{Remarks}

\newcommand{\R}{\mathbb{R}}
\newcommand{\one}{\mathbf{1}}
\newcommand{\cof}{\operatorname{cof}}
\newcommand{\diag}{\operatorname{diag}}
\newcommand{\tr}{\operatorname{tr}}
\newcommand{\FF}{\mathcal{F}}

\begin{document}

\begin{frontmatter}

\title{Principal minors of effective-resistance matrices and local resistance radii}

\author[ytu]{Guangfu Wang\corref{cor1}}
\ead{gfwang@ytu.edu.cn}
\cortext[cor1]{Corresponding author.}

\address[ytu]{School of Mathematics and Information Sciences, Yantai University, Yantai, Shandong Province 264005, China}

\begin{abstract}
Let $G$ be a finite connected weighted graph and let $R$ be its effective-resistance matrix.  For every nonempty vertex set $S$, we factor the cofactor sum and determinant of the principal resistance submatrix $R[S]$ into an enumerative term and a boundary potential-theoretic term.  If $\tau(G)$ is the weighted spanning tree enumerator and $\kappa_G(S)$ is the weighted enumerator of $S$-rooted spanning forests, then
\[
        \cof R[S]=(-2)^{|S|-1}\kappa_G(S)/\tau(G).
\]
After Kron reduction to $S$, with reduced Laplacian $K=L^S$, $Q=K^+$, and $q=\diag(Q)$, the remaining normalized factor is
\[
        \det R[S]/\cof R[S]
        =\frac{2}{|S|}\tr Q+\frac12 q^{\mathsf T}Kq.
\]
The cofactor factor is a principal specialization of known resistance-minor identities; the contribution here is the boundary/Kron-reduction factorization and local radius calculus.  Equivalently, the normalized factor is the maximum of $u^{\mathsf T}R[S]u$ over all $u\in\R^S$ satisfying $\one^{\mathsf T}u=1$.  This optimization viewpoint yields monotonicity under enlargement of $S$, an exact one-point update formula, and a support criterion for equality.  Small star examples show that the resulting set function is neither submodular nor supermodular in general.
\end{abstract}


\end{frontmatter}

\section{Introduction}\label{sec:introduction}

Effective resistance is a central invariant of a finite weighted graph.  It was introduced through electrical network theory: when an edge of conductance $c_e$ is interpreted as a resistor of resistance $1/c_e$, the effective resistance between two vertices is the voltage drop required to send one unit of current between them.  The same quantity also appears in random walks, spectral graph theory, numerical linear algebra, network analysis, and mathematical chemistry.  If $L$ is the weighted Laplacian of a connected graph and $L^+$ is its Moore--Penrose pseudoinverse, then
\[
        R_{ij}=(e_i-e_j)^{\mathsf T}L^+(e_i-e_j)
\]
defines the effective resistance between vertices $i$ and $j$.  The matrix $R=(R_{ij})$ is the effective-resistance matrix.  Its entries form a metric, while its matrix structure records more information than the metric alone, including Laplacian pseudoinverse data and the Euclidean simplex geometry associated with the graph; see, for example, \cite{KleinRandic1993,Bapat2010,Devriendt2022LAA,Devriendt2022Thesis}.

The present paper studies principal submatrices $R[S]$, where $S$ is a nonempty set of vertices.  Such submatrices arise naturally when only a boundary set is observed, or when one wants to understand how global resistance geometry changes after restricting to a prescribed group of vertices.  Their determinants and cofactor sums are also natural linear-algebraic invariants: the cofactor sum measures the restriction of the quadratic form to the zero-sum hyperplane, while the ratio $\det R[S]/\cof R[S]$ is a resistance-radius type quantity.  Thus principal resistance minors connect three viewpoints that are usually treated separately: forest enumeration, Schur complements, and concave quadratic optimization.

A classical starting point is the determinant formula of Graham and Pollak \cite{GrahamPollak1971}.  If $G$ is a tree on $n$ vertices and $D$ is its path-distance matrix, then
\[
        \det D=(-1)^{n-1}2^{n-2}(n-1).
\]
The independence of this determinant from the shape of the tree motivated a substantial literature on distance matrices and Laplacians.  Related normalized determinant ratios already appear in the block-decomposition formula of Graham, Hoffman, and Hosoya \cite{GrahamHoffmanHosoya1977}, and the characteristic-polynomial viewpoint for tree distance matrices was developed by Graham and Lov\'{a}sz \cite{GrahamLovasz1978}.  A weighted version for tree distance matrices was proved by Bapat, Kirkland, and Neumann \cite{BapatKirklandNeumann2005}.  More recently, Richman, Shokrieh, and Wu \cite{RichmanShokriehWuPrincipal} obtained formulas for all principal minors $D[S]$ of a tree distance matrix in terms of rooted spanning forests, together with a monotonicity theorem for the normalized principal minors $\det D[S]/\cof D[S]$.  Guti\'{e}rrez and Lillo \cite{GutierrezLillo2024} gave a related combinatorial treatment of principal tree-distance minors.  Since tree distance is exactly effective resistance on a tree, it is natural to ask which parts of this picture persist for arbitrary weighted networks.

For general graphs, cycles change the nature of the answer.  Identities for minors of Laplacian, resistance, and distance matrices were obtained by Bapat and Sivasubramanian \cite{BapatSivasubramanian2011}, and were extended to graphs with arbitrary weights by Ali, Atik, and Bapat \cite{AliAtikBapat2020}.  Bapat's formulas for the determinant and inverse of a full resistance matrix \cite{Bapat2004}, as well as recent determinant formulas of Richman, Shokrieh, and Wu \cite{RichmanShokriehWuResistanceDet}, are directly relevant.  In this setting, Kron reduction provides the natural way to separate boundary data from interior graph structure, in accordance with the electrical principle that effective resistances on a boundary are preserved by Schur-complement elimination of interior vertices; see \cite{DorflerBullo2013,Devriendt2022LAA}.

\subsection{Relation to previous work and novelty}\label{subsec:related}

The cofactor identity used below should be read as a principal-submatrix specialization of known cofactor-sum identities for resistance matrices, especially those of Bapat and Sivasubramanian and their weighted extension by Ali, Atik, and Bapat \cite{BapatSivasubramanian2011,AliAtikBapat2020}.  We nevertheless include a short proof because the Kron-reduction argument fixes the normalization and exposes the boundary forest factor that is needed for the determinant factorization.

The normalized factor $\det R[S]/\cof R[S]$ belongs to the resistance-radius framework developed by Devriendt and Devriendt--Lambiotte \cite{Devriendt2022Thesis,DevriendtLambiotte2022}.  The same optimization perspective also appears in the tree-distance setting in the work of Richman, Shokrieh, and Wu \cite{RichmanShokriehWuPrincipal}.  Their resistance-determinant preprint also points to closely related principal-minor questions \cite{RichmanShokriehWuResistanceDet}.  Thus the main contribution here is not the cofactor formula alone.  Rather, the paper combines the cofactor/minor identities with Kron reduction and the resistance-radius variational formula to obtain a boundary factorization for arbitrary weighted networks, together with an exact one-point update formula and support criterion for equality in the resulting monotonicity statement.

\subsection{Main results and contribution}

Let $G=(V,E,c)$ be a finite connected graph with positive edge conductances $c_e$.  Let
\[
        \tau(G)=\sum_T\prod_{e\in T}c_e
\]
be the weighted spanning tree enumerator, where the sum is over all spanning trees of $G$.  For a nonempty set $S\subseteq V$, let $\kappa_G(S)$ be the weighted enumerator of $S$-rooted spanning forests, namely
\[
        \kappa_G(S)=\sum_{F\in\FF(G;S)}\prod_{e\in F}c_e,
\]
where each component of $F$ contains exactly one vertex of $S$.  Let $L^S$ denote the Kron-reduced Laplacian obtained from $L$ by eliminating $V\setminus S$.

\begin{theorem}\label{thm:intro-main}
For every nonempty $S\subseteq V$,
\begin{equation}\label{eq:intro-cof}
        \cof R[S]=(-2)^{|S|-1}\frac{\kappa_G(S)}{\tau(G)},
\end{equation}
and, writing $K=L^S$, $Q=K^+$, and $q=\diag(Q)$,
\begin{equation}\label{eq:intro-det}
        \det R[S]
        =(-2)^{|S|-1}\frac{\kappa_G(S)}{\tau(G)}
        \left(\frac{2}{|S|}\tr Q+\frac12 q^{\mathsf T}Kq\right).
\end{equation}
Here $\cof$ denotes the sum of all cofactors.  When $|S|=1$, the factor in parentheses in \eqref{eq:intro-det} is interpreted as $0$.
\end{theorem}

Theorem~\ref{thm:intro-main} should be interpreted as a boundary factorization: the first identity is compatible with the known minor identities cited above, while the proof below derives it through Kron reduction.  The determinant then separates into two parts.  The factor $(-2)^{|S|-1}\kappa_G(S)/\tau(G)$ is purely enumerative.  The second factor is intrinsic to the reduced boundary network.  Define
\[
        \rho_G(S):=\frac{\det R[S]}{\cof R[S]}.
\]
Then
\begin{equation}\label{eq:intro-rho}
        \rho_G(S)=\frac{2}{|S|}\tr(L^S)^+
        +\frac12 q^{\mathsf T}L^Sq,
        \qquad q=\diag((L^S)^+),
\end{equation}
and this number has the variational interpretation
\begin{equation}\label{eq:intro-opt}
        \rho_G(S)=\max\{u^{\mathsf T}R[S]u:\; u\in\R^S,\; \one^{\mathsf T}u=1\}.
\end{equation}
In the terminology of resistance geometry, this is a resistance-radius type quantity \cite{Devriendt2022Thesis,DevriendtLambiotte2022}.  For trees, the formula is compatible with the normalized-minor factor in \cite{RichmanShokriehWuPrincipal}; the additional tree-specific feature there is a fully combinatorial expansion of this factor.

The optimization formula also gives a local calculus for the set function $S\mapsto\rho_G(S)$.  It immediately implies monotonicity: if $\varnothing\neq A\subseteq B\subseteq V$, then $\rho_G(A)\le \rho_G(B)$.  More precisely, for $x\notin S$, put $A=R[S]$ and $b=R[S,x]$.  If $|S|\ge2$, define
\[
        a=b^{\mathsf T}A^{-1}\one,
        \qquad
        \eta=b^{\mathsf T}A^{-1}b.
\]
Then $\eta>0$ and
\begin{equation}\label{eq:intro-update}
        \frac{1}{\rho_G(S\cup\{x\})}
        =\frac{1}{\rho_G(S)}-\frac{(a-1)^2}{\eta}.
\end{equation}
Consequently, equality in the one-point monotonicity step is equivalent to $a=1$.  If $u_S$ is the unique optimizer in \eqref{eq:intro-opt}, this equality condition can be written as
\[
        \sum_{s\in S}u_S(s)R_{xs}=\rho_G(S).
\]
For arbitrary $A\subseteq B$, equality $\rho_G(A)=\rho_G(B)$ holds exactly when the optimizer for $B$ is supported on $A$.  Finally, small star examples show that no unconditional submodularity or supermodularity law holds for $S\mapsto\rho_G(S)$.  The correct substitute is therefore the exact local formula \eqref{eq:intro-update} and the support criterion for equality.

\subsection{Organization}

\Cref{sec:preliminaries} recalls weighted Laplacians, effective resistance, rooted forests, cofactor sums, and Kron reduction.  \Cref{sec:principal} proves the cofactor formula, the Kron-reduced expression for $\rho_G(S)$, and the determinant formula for $R[S]$.  \Cref{sec:radius} proves the optimization theorem, monotonicity, equality criteria, and the one-point update formula.  \Cref{sec:examples} gives examples, including the failure of submodularity and supermodularity.

\section{Preliminaries}\label{sec:preliminaries}

Throughout the paper, $G=(V,E,c)$ is a finite connected undirected graph with positive edge conductances $c_e$.  Parallel edges are allowed; loops play no role in effective resistance and are omitted.  We write $n=|V|$.  After choosing an arbitrary orientation of the edges, let $B\in\R^{V\times E}$ be the incidence matrix and let $C=\diag(c_e)$.  The weighted Laplacian is
\[
        L=BCB^{\mathsf T}.
\]
It is symmetric positive semidefinite, satisfies $L\one=0$, and has kernel $\R\one$ because $G$ is connected.  Its Moore--Penrose pseudoinverse is denoted $L^+$.

For a finite set $X$, the notation $\R^X$ means real-valued column vectors indexed by $X$.  The all-ones vector $\one$ always has the dimension determined by context, and $\one^\perp$ denotes the subspace of vectors whose coordinates sum to zero.  If a matrix $M$ is indexed by a finite set $X$ and $A,B\subseteq X$, then $M[A,B]$ denotes the submatrix with rows in $A$ and columns in $B$, and $M[A]=M[A,A]$.  In block formulas we also write $M_{AB}$ for $M[A,B]$.  For $u\in\R^S$, the symbol $\widetilde u\in\R^V$ denotes extension by zero outside $S$.

\subsection{Effective resistance matrices}

The effective resistance between vertices $i,j\in V$ is
\[
        R_{ij}=(e_i-e_j)^{\mathsf T}L^+(e_i-e_j),
\]
where $e_i$ and $e_j$ are the corresponding standard coordinate vectors.
Let $z=\diag(L^+)$.  Then the full resistance matrix can be written as
\begin{equation}\label{eq:R-from-Q}
        R=z\one^{\mathsf T}+\one z^{\mathsf T}-2L^+.
\end{equation}
For $S\subseteq V$, we write $R[S]$ for the principal submatrix indexed by $S$.  We also use the following convention: if $A$ is a square matrix, then
\[
        \cof A=\sum_{i,j}(-1)^{i+j}\det A(i|j),
\]
where $A(i|j)$ is obtained from $A$ by deleting row $i$ and column $j$.  Thus, if $A$ is invertible,
\begin{equation}\label{eq:cof-inv}
        \cof A=(\det A)\one^{\mathsf T}A^{-1}\one.
\end{equation}
For a $1\times 1$ matrix, the cofactor sum is the determinant of the empty matrix, hence equals $1$.

\begin{lemma}\label{lem:negative-type}
Let $S\subseteq V$ and $|S|=k\geq 2$.  Then $R[S]$ has inertia $(1,k-1,0)$: one positive eigenvalue, $k-1$ negative eigenvalues, and no zero eigenvalues.  Moreover,
\begin{equation}\label{eq:negative-type}
        u^{\mathsf T}R[S]u=-2\widetilde u^{\mathsf T}L^+\widetilde u<0
\end{equation}
for every nonzero $u\in\R^S$ satisfying $\one^{\mathsf T}u=0$, where $\widetilde u\in\R^V$ is the extension by zero outside $S$.
\end{lemma}

\begin{proof}
If $\one^{\mathsf T}u=0$, then the two rank-one terms in \eqref{eq:R-from-Q} vanish on $u$, and \eqref{eq:negative-type} follows.  Since $\widetilde u\in\one^\perp$ and $L^+$ is positive definite on $\one^\perp$, the inequality is strict for $u\neq0$.

Thus the quadratic form defined by $R[S]$ is negative definite on the codimension-one subspace $\one^\perp\subset\R^S$.  Hence $R[S]$ has at least $k-1$ negative eigenvalues and at most one nonnegative eigenvalue.  Since the diagonal entries of $R[S]$ are zero, its trace is zero; and since $R[S]$ is not the zero matrix for $k\ge2$, the remaining eigenvalue is positive.  This proves the stated inertia.
\end{proof}

\subsection{Rooted forests and the matrix-tree theorem}

Let $\tau(G)$ be the weighted spanning tree enumerator
\[
        \tau(G)=\sum_{T}\prod_{e\in T}c_e.
\]
Equivalently, by Kirchhoff's matrix-tree theorem, $\tau(G)$ is any diagonal cofactor of $L$.

For a nonempty $S\subseteq V$, an $S$-rooted spanning forest is a spanning forest in which each connected component contains exactly one vertex of $S$.  Let $\FF(G;S)$ be the set of all such forests and define
\[
        \kappa_G(S)=\sum_{F\in\FF(G;S)}\prod_{e\in F}c_e.
\]
The all-minors form of the matrix-tree theorem gives
\begin{equation}\label{eq:kappa-det}
        \kappa_G(S)=\det L[V\setminus S,V\setminus S],
\end{equation}
with the convention that the determinant of the empty matrix is $1$; see, for example, \cite[Chapter VI]{Tutte2001} or \cite{Bapat2010}.

\subsection{Kron reduction}\label{subsec:kron}

Fix a nonempty $S\subseteq V$ and set $I=V\setminus S$.  If $I\neq\varnothing$, the principal block $L_{II}$ is positive definite.  The Kron-reduced Laplacian on $S$ is the Schur complement
\begin{equation}\label{eq:kron-def}
        L^S=L_{SS}-L_{SI}L_{II}^{-1}L_{IS}.
\end{equation}
If $I=\varnothing$, we set $L^S=L$.  The matrix $L^S$ is again a Laplacian matrix on the vertex set $S$: it is symmetric, has zero row sums, has nonpositive off-diagonal entries, and has rank $|S|-1$.  Kron reduction is also called boundary reduction or Schur-complement reduction in electrical network theory.

\begin{lemma}\label{lem:kron-preserve}
Let $K=L^S$.  For every $x\in\R^S$ with $\one^{\mathsf T}x=0$,
\begin{equation}\label{eq:kron-energy}
        x^{\mathsf T}K^+x=\widetilde x^{\mathsf T}L^+\widetilde x,
\end{equation}
where $\widetilde x$ is the extension of $x$ by zero to $V$.  Consequently, the effective resistance matrix of the Kron-reduced network $K$ is exactly $R[S]$.
\end{lemma}

\begin{proof}
If $I=\varnothing$, then $K=L$ and there is nothing to prove.  Assume $I\neq\varnothing$.  We give the standard Schur-complement argument.  Let $x\in\R^S$ with $\one^{\mathsf T}x=0$.  Choose a boundary potential $v_S$ satisfying $Kv_S=x$, which is unique up to addition of a constant.  Define the interior potential by
\[
        v_I=-L_{II}^{-1}L_{IS}v_S.
\]
Then the vector $v=(v_S,v_I)$ satisfies $Lv=\widetilde x$.  Since $\widetilde x\perp\one$, the vector $L^+\widetilde x$ is another solution of $Ly=\widetilde x$ up to addition of a constant vector.  Similarly, $v_S$ and $K^+x$ solve $Ky=x$ and differ by a constant vector.  As both $x$ and $\widetilde x$ have zero coordinate sum, these constants do not affect the energy pairing, and hence
\[
        \widetilde x^{\mathsf T}L^+\widetilde x
        =\widetilde x^{\mathsf T}v
        =x^{\mathsf T}v_S
        =x^{\mathsf T}K^+x.
\]
Taking $x=e_a-e_b$ for $a,b\in S$ proves preservation of effective resistances.  This is the usual electrical interpretation of Kron reduction; see also \cite{DorflerBullo2013,Devriendt2022LAA}.
\end{proof}

\begin{lemma}\label{lem:kron-tau}
Let $K=L^S$ and let $\tau(K)$ denote any diagonal cofactor of $K$, with the convention $\tau(K)=1$ if $|S|=1$.  Then
\begin{equation}\label{eq:kron-tau}
        \tau(K)=\frac{\tau(G)}{\kappa_G(S)}.
\end{equation}
\end{lemma}

\begin{proof}
If $|S|=1$, then \eqref{eq:kappa-det} gives $\kappa_G(S)=\tau(G)$, so the statement follows from the convention $\tau(K)=1$.  If $S=V$, then $K=L$ and $\kappa_G(V)=1$, so the assertion is also immediate.

Assume now that $|S|\ge2$ and $I\neq\varnothing$.  Choose $r\in S$.  Delete the row and column indexed by $r$ from $L$ and order the remaining indices as $(S\setminus\{r\})\sqcup I$.  The determinant of this matrix is $\tau(G)$.  Taking the Schur complement of the positive definite block $L_{II}$ gives
\[
        \tau(G)=\det(L_{II})\det K[S\setminus\{r\},S\setminus\{r\}].
\]
By \eqref{eq:kappa-det}, $\det(L_{II})=\kappa_G(S)$; by the matrix-tree theorem applied to the reduced Laplacian $K$, the second determinant is $\tau(K)$.  This proves \eqref{eq:kron-tau}.
\end{proof}

\section{Principal minors of resistance matrices}\label{sec:principal}

We begin with a general linear-algebra identity for cofactor sums.

\begin{lemma}\label{lem:cof-restriction}
Let $A$ be a real symmetric $k\times k$ matrix and let $H$ be any $k\times(k-1)$ matrix whose columns form an orthonormal basis of $\one^\perp$.  Then
\begin{equation}\label{eq:cof-restriction}
        \cof A=k\det(H^{\mathsf T}AH).
\end{equation}
\end{lemma}

\begin{proof}
First assume $A$ is invertible.  Let $U=[k^{-1/2}\one\; H]$, an orthogonal matrix, and write
\[
        U^{\mathsf T}AU=\begin{pmatrix} \alpha & \beta^{\mathsf T}\\ \beta & C\end{pmatrix},
        \qquad C=H^{\mathsf T}AH.
\]
Using \eqref{eq:cof-inv},
\[
        \cof A=(\det A)\one^{\mathsf T}A^{-1}\one
        =k\det(U^{\mathsf T}AU)\bigl((U^{\mathsf T}AU)^{-1}\bigr)_{11}.
\]
For an invertible block matrix, the product of the determinant and the $(1,1)$ entry of the inverse is the complementary principal minor, namely $\det C$.  Thus $\cof A=k\det C$.  Both sides of \eqref{eq:cof-restriction} are polynomial functions of the entries of $A$, and the invertible matrices are dense among all real symmetric matrices.  The identity therefore extends to singular $A$ by continuity.
\end{proof}

The following full-matrix cofactor identity is well known; for example, it follows as a special case of the cofactor-sum identity of Bapat and Sivasubramanian and of the weighted extension by Ali, Atik, and Bapat \cite{BapatSivasubramanian2011,AliAtikBapat2020}.  We include the proof to make the normalization explicit.

\begin{proposition}\label{prop:full-cof}
Let $M$ be the Laplacian of a connected weighted graph on $k$ vertices and let $\Omega$ be its effective resistance matrix.  Then
\begin{equation}\label{eq:full-cof}
        \cof \Omega=\frac{(-2)^{k-1}}{\tau(M)},
\end{equation}
where $\tau(M)$ is any diagonal cofactor of $M$, with $\tau(M)=1$ for $k=1$.
\end{proposition}

\begin{proof}
The case $k=1$ is immediate: $\Omega=[0]$ and $\cof\Omega=1$.

Assume $k\ge2$.  Let $H$ be an orthonormal basis matrix for $\one^\perp$.  If $Q=M^+$, then for vectors in $\one^\perp$ the rank-one terms in the resistance representation vanish, so
\[
        H^{\mathsf T}\Omega H=-2H^{\mathsf T}QH.
\]
The restriction of $Q$ to $\one^\perp$ is the inverse of the restriction of $M$ to $\one^\perp$.  Therefore
\[
        \det(H^{\mathsf T}QH)=\frac{1}{\det(H^{\mathsf T}MH)}.
\]
The determinant $\det(H^{\mathsf T}MH)$ is the product of the nonzero eigenvalues of $M$, hence equals $k\tau(M)$ by Kirchhoff's theorem.  Applying Lemma~\ref{lem:cof-restriction} gives
\[
        \cof\Omega=k(-2)^{k-1}\det(H^{\mathsf T}QH)
        =\frac{k(-2)^{k-1}}{k\tau(M)},
\]
which is \eqref{eq:full-cof}.
\end{proof}

The next formula is likewise a principal-submatrix specialization of known resistance-matrix minor identities \cite{BapatSivasubramanian2011,AliAtikBapat2020}.  The proof emphasizes the boundary reduction that will also be used for the normalized factor.

\begin{theorem}\label{thm:cof-principal}
For every nonempty $S\subseteq V$,
\begin{equation}\label{eq:cof-principal}
        \cof R[S]=(-2)^{|S|-1}\frac{\kappa_G(S)}{\tau(G)}.
\end{equation}
\end{theorem}

\begin{proof}
Let $K=L^S$ be the Kron-reduced Laplacian on $S$.  By Lemma~\ref{lem:kron-preserve}, the matrix $R[S]$ is the full effective resistance matrix of the network with Laplacian $K$.  Applying Proposition~\ref{prop:full-cof} to $K$ and then using Lemma~\ref{lem:kron-tau}, we obtain
\[
        \cof R[S]=\frac{(-2)^{|S|-1}}{\tau(K)}
        =(-2)^{|S|-1}\frac{\kappa_G(S)}{\tau(G)}.
\]
\end{proof}

The determinant itself is obtained by multiplying \eqref{eq:cof-principal} by the normalized factor $\rho_G(S)$.  We first record the elementary optimization identity that identifies this factor.

\begin{lemma}\label{lem:normalized-optimization}
Let $A$ be a real symmetric $k\times k$ matrix.  If $k=1$, assume $A=[0]$ and set $\rho(A)=0$.  If $k\ge2$, assume that $A$ is invertible and that the quadratic form $u\mapsto u^{\mathsf T}Au$ is negative definite on $\one^\perp$, and set
\[
        \rho(A)=\frac{\det A}{\cof A}.
\]
Then
\[
        \rho(A)=\max\{u^{\mathsf T}Au:\;u\in\R^k,\;\one^{\mathsf T}u=1\}.
\]
For $k\ge2$, the maximizer is unique and equals $\rho(A)A^{-1}\one$.
\end{lemma}

\begin{proof}
The case $k=1$ is immediate.  Assume $k\ge2$.  If $H$ is an orthonormal basis matrix for $\one^\perp$, then $H^{\mathsf T}AH$ is negative definite; hence Lemma~\ref{lem:cof-restriction} gives
\[
        \cof A=k\det(H^{\mathsf T}AH)\neq0.
\]
The objective function is strictly concave on the affine hyperplane $\one^{\mathsf T}u=1$, because its Hessian restricts to a negative definite form on $\one^\perp$.  Hence any critical point is the unique global maximum.

The Lagrange multiplier equation is $Au=\lambda\one$.  Since $A$ is invertible,
\[
        u=\lambda A^{-1}\one.
\]
The constraint gives $1=\lambda\one^{\mathsf T}A^{-1}\one$.  By \eqref{eq:cof-inv},
\[
        \lambda=\frac{1}{\one^{\mathsf T}A^{-1}\one}
        =\frac{\det A}{\cof A}=\rho(A).
\]
At the critical point, $u^{\mathsf T}Au=\lambda\one^{\mathsf T}u=\lambda$, proving the claim.
\end{proof}

We now compute $\rho_G(S)$ in terms of the Kron reduction.

\begin{proposition}\label{prop:rho-kron}
Let $S\subseteq V$ be nonempty, put $k=|S|$, and let $K=L^S$.  Let $Q=K^+$ and $q=\diag(Q)$.  Then
\begin{equation}\label{eq:rho-kron}
        \rho_G(S):=\frac{\det R[S]}{\cof R[S]}
        =\frac{2}{k}\tr Q+\frac12 q^{\mathsf T}Kq.
\end{equation}
For $k=1$, both sides are $0$.
\end{proposition}

\begin{proof}
The case $k=1$ is immediate.  Assume $k\ge2$.  Since $R[S]$ is the effective resistance matrix of the Kron-reduced network, we may work entirely on the vertex set $S$.  Thus
\[
        R[S]=q\one^{\mathsf T}+\one q^{\mathsf T}-2Q.
\]
For any $u\in\R^S$ with $\one^{\mathsf T}u=1$, write
\[
        u=\frac1k\one+z,
        \qquad z\in\one^\perp.
\]
Then
\begin{equation}\label{eq:objective-z}
        u^{\mathsf T}R[S]u
        =\frac{2}{k}\one^{\mathsf T}q+2q^{\mathsf T}z-2z^{\mathsf T}Qz.
\end{equation}
The quadratic form is strictly concave in $z$ because $Q$ is positive definite on $\one^\perp$.  The unique critical point satisfies
\[
        Qz=\frac12\left(q-\frac{\one^{\mathsf T}q}{k}\one\right).
\]
Since $QK=KQ$ is the orthogonal projection onto $\one^\perp$, the solution is
\begin{equation}\label{eq:z-star}
        z_* =\frac12Kq.
\end{equation}
Substituting \eqref{eq:z-star} into \eqref{eq:objective-z} gives
\[
        \max_{\one^{\mathsf T}u=1}u^{\mathsf T}R[S]u
        =\frac{2}{k}\tr Q+\frac12 q^{\mathsf T}Kq.
\]
By Lemmas~\ref{lem:negative-type} and \ref{lem:normalized-optimization}, this maximum is $\det R[S]/\cof R[S]$.  This proves the formula.  The same computation also shows that the maximizing vector is $k^{-1}\one+\frac12Kq$.
\end{proof}

The vector $k^{-1}\one+\frac12Kq$ is the boundary-network version of the resistance-curvature vector used in resistance geometry \cite{Devriendt2022Thesis,DevriendtLambiotte2022}.

\begin{corollary}\label{cor:rho-average}
Let $S\subseteq V$ be nonempty with $k=|S|\ge2$, and let $K=L^S$, $Q=K^+$, and $q=\diag(Q)$.  Then
\begin{equation}\label{eq:rho-average}
        \rho_G(S)=\frac{1}{k^2}\sum_{a,b\in S}R_{ab}+\frac12 q^{\mathsf T}Kq.
\end{equation}
In particular,
\begin{equation}\label{eq:rho-lower-average}
        \rho_G(S)\ge \frac{1}{k^2}\sum_{a,b\in S}R_{ab},
\end{equation}
with equality if and only if the diagonal entries of $Q$ are all equal.
\end{corollary}

\begin{proof}
For the Kron-reduced network, $R[S]=q\one^{\mathsf T}+\one q^{\mathsf T}-2Q$ and $Q\one=0$.  Hence
\[
        \sum_{a,b\in S}R_{ab}=2k\tr Q.
\]
Substituting this identity in \eqref{eq:rho-kron} gives \eqref{eq:rho-average}.  Since $K$ is a connected Laplacian, $q^{\mathsf T}Kq\ge0$, with equality exactly when $q\in\ker K=\R\one$.
\end{proof}

Combining the known cofactor factor in Theorem~\ref{thm:cof-principal} with the Kron-reduced radius formula in Proposition~\ref{prop:rho-kron} gives the boundary factorization of the determinant.

\begin{theorem}\label{thm:det-principal}
Let $S\subseteq V$ be nonempty and set $k=|S|$.  Let $K=L^S$, $Q=K^+$, and $q=\diag(Q)$.  Then
\begin{equation}\label{eq:det-principal}
        \det R[S]
        =(-2)^{k-1}\frac{\kappa_G(S)}{\tau(G)}
        \left(\frac{2}{k}\tr Q+\frac12 q^{\mathsf T}Kq\right),
\end{equation}
with the convention that the parenthesized factor is zero for $k=1$.
\end{theorem}

\begin{proof}
For $k=1$, $R[S]=[0]$ and \eqref{eq:det-principal} is clear.  For $k\ge2$, multiply the cofactor formula \eqref{eq:cof-principal} by \eqref{eq:rho-kron}.
\end{proof}

\begin{remark}\label{rem:trees}
Suppose $G$ is a tree with edge resistances $r_e=1/c_e$.  Then effective resistance is path distance.  Since $\tau(G)=\prod_{e\in E}c_e$, the factor $\kappa_G(S)/\tau(G)$ becomes
\[
        \sum_{F\in\FF(G;S)}\prod_{e\notin F}r_e,
\]
which is the coweight enumerator appearing in formulas for tree distance submatrices.  Thus Theorem~\ref{thm:cof-principal} recovers the cofactor-sum identity for tree distance principal submatrices.  The determinant formula above recovers the same determinant after inserting the corresponding normalized factor $\rho_G(S)$; the tree-specific work of \cite{RichmanShokriehWuPrincipal} gives an additional rooted-forest expression for that factor.
\end{remark}

\section{Normalized principal minors and local radii}\label{sec:radius}

In this section we study
\[
        \rho_G(S)=\frac{\det R[S]}{\cof R[S]}.
\]
By Lemma~\ref{lem:negative-type}, $R[S]$ is invertible for $|S|\ge2$, and Theorem~\ref{thm:cof-principal} shows that $\cof R[S]\neq0$ for every nonempty $S$.  Moreover, Proposition~\ref{prop:rho-kron} gives $\rho_G(S)>0$ whenever $|S|\ge2$, while $\rho_G(S)=0$ for $|S|=1$.

The optimization viewpoint below is closely related to the resistance-radius formulation in Devriendt's resistance geometry and to the normalized-principal-minor optimization used by Richman, Shokrieh, and Wu in the tree case \cite{Devriendt2022Thesis,RichmanShokriehWuPrincipal}.

\begin{proposition}\label{prop:optimization}
For every nonempty $S\subseteq V$,
\begin{equation}\label{eq:optimization}
        \rho_G(S)=\max\{u^{\mathsf T}R[S]u:\;u\in\R^S,\;\one^{\mathsf T}u=1\}.
\end{equation}
For $|S|\ge2$, the maximizer is unique and equals
\begin{equation}\label{eq:u-star-inverse}
        u_S=\rho_G(S)R[S]^{-1}\one.
\end{equation}
Equivalently, if $K=L^S$, $Q=K^+$, and $q=\diag(Q)$, then
\begin{equation}\label{eq:u-star-kron}
        u_S=\frac1{|S|}\one+\frac12 Kq.
\end{equation}
\end{proposition}

\begin{proof}
If $|S|=1$, then $R[S]=[0]$, $\rho_G(S)=0$, and the statement is immediate.  If $|S|\ge2$, Lemma~\ref{lem:negative-type} shows that $R[S]$ is invertible and negative definite on $\one^\perp$.  Applying Lemma~\ref{lem:normalized-optimization} to $A=R[S]$ gives \eqref{eq:optimization} and \eqref{eq:u-star-inverse}.  Formula \eqref{eq:u-star-kron} is the maximizing vector computed in the proof of Proposition~\ref{prop:rho-kron}.
\end{proof}

\begin{corollary}\label{cor:monotonicity}
If $\varnothing\neq A\subseteq B\subseteq V$, then
\begin{equation}\label{eq:monotonicity}
        \rho_G(A)\le \rho_G(B).
\end{equation}
\end{corollary}

\begin{proof}
Extend vectors in $\R^A$ by zero to vectors in $\R^B$.  The feasible set for the maximization defining $\rho_G(A)$ is then contained in the feasible set defining $\rho_G(B)$.
\end{proof}

Monotonicity of the corresponding resistance-radius quantity already appears in Devriendt's framework and, for tree distance matrices, in Richman, Shokrieh, and Wu \cite{Devriendt2022Thesis,RichmanShokriehWuPrincipal}.  The point of the next results is the exact equality criterion and one-point update formula in the present boundary notation.

The next result gives the exact equality condition in Corollary~\ref{cor:monotonicity}.

\begin{theorem}\label{thm:equality-monotonicity}
Let $\varnothing\neq A\subseteq B\subseteq V$.  Let $u_B$ be the unique maximizer for $B$ in \eqref{eq:optimization} if $|B|\ge2$, and let $u_B=1$ if $|B|=1$.  Then
\begin{equation}\label{eq:equality-support}
        \rho_G(A)=\rho_G(B)
        \quad\Longleftrightarrow\quad
        u_B(v)=0\;\text{ for every }v\in B\setminus A.
\end{equation}
\end{theorem}

\begin{proof}
The statement is trivial when $A=B$, so assume $A\subsetneq B$.  Let $\widehat u_A\in\R^B$ be the extension of the optimizer $u_A$ by zero outside $A$.  Then $\widehat u_A$ is feasible for the $B$-problem and has objective value $\rho_G(A)$.

If $\rho_G(A)=\rho_G(B)$, then $\widehat u_A$ is also a maximizer for the $B$-problem.  By strict concavity on the feasible hyperplane, the maximizer is unique; hence $u_B=\widehat u_A$, and $u_B$ vanishes on $B\setminus A$.  Conversely, if $u_B$ is supported on $A$, then its restriction to $A$ is feasible for the $A$-problem and has objective value $\rho_G(B)$.  Hence $\rho_G(A)\ge\rho_G(B)$, while the reverse inequality follows from monotonicity.
\end{proof}

\subsection{One-point update formula}

We now give an exact formula for adding a single vertex.

\begin{theorem}\label{thm:one-point}
Let $S\subseteq V$ with $|S|\ge2$, and let $x\in V\setminus S$.  Put
\[
        A=R[S],\qquad b=R[S,x]\in\R^S,
\]
where $b_s=R_{sx}$.  Define
\[
        a=b^{\mathsf T}A^{-1}\one,
        \qquad
        \eta=b^{\mathsf T}A^{-1}b.
\]
Then $\eta>0$ and
\begin{equation}\label{eq:one-point-inverse}
        \frac1{\rho_G(S\cup\{x\})}
        =\frac1{\rho_G(S)}-\frac{(a-1)^2}{\eta}.
\end{equation}
Equivalently,
\begin{equation}\label{eq:one-point-rho}
        \rho_G(S\cup\{x\})
        =\frac{\rho_G(S)\eta}{\eta-\rho_G(S)(a-1)^2}.
\end{equation}
Moreover, if $u_{S\cup\{x\}}$ denotes the optimizer for $S\cup\{x\}$, then
\begin{equation}\label{eq:new-coordinate}
        u_{S\cup\{x\}}(x)
        =\rho_G(S\cup\{x\})\frac{a-1}{\eta}.
\end{equation}
\end{theorem}

\begin{proof}
With the ordering in which $x$ is last,
\[
        R[S\cup\{x\}]=
        \begin{pmatrix}
        A & b\\
        b^{\mathsf T} & 0
        \end{pmatrix}.
\]
By Lemma~\ref{lem:negative-type}, $A$ has inertia $(1,|S|-1,0)$ and $R[S\cup\{x\}]$ has inertia $(1,|S|,0)$.  The Schur complement of $A$ in the displayed block matrix is $-b^{\mathsf T}A^{-1}b=-\eta$.  Inertia additivity for Schur complements therefore implies $-\eta<0$, hence $\eta>0$.

The block inverse formula gives
\[
        R[S\cup\{x\}]^{-1}=\begin{pmatrix}
        A^{-1}-\eta^{-1}A^{-1}bb^{\mathsf T}A^{-1} & \eta^{-1}A^{-1}b\\
        \eta^{-1}b^{\mathsf T}A^{-1} & -\eta^{-1}
        \end{pmatrix}.
\]
Summing all entries yields
\[
\begin{aligned}
        \one^{\mathsf T}R[S\cup\{x\}]^{-1}\one
        &=\one^{\mathsf T}A^{-1}\one
          -\frac{(b^{\mathsf T}A^{-1}\one)^2}{\eta}
          +2\frac{b^{\mathsf T}A^{-1}\one}{\eta}
          -\frac1\eta \\
        &=\frac1{\rho_G(S)}-\frac{(a-1)^2}{\eta}.
\end{aligned}
\]
Using $\one^{\mathsf T}R[T]^{-1}\one=1/\rho_G(T)$ proves \eqref{eq:one-point-inverse}.  The formula \eqref{eq:one-point-rho} is an algebraic rearrangement.

For the final statement, the last component of $R[S\cup\{x\}]^{-1}\one$ is $(a-1)/\eta$.  Multiplying by $\rho_G(S\cup\{x\})$, as in \eqref{eq:u-star-inverse}, gives \eqref{eq:new-coordinate}.
\end{proof}

\begin{corollary}\label{cor:one-point-equality}
Let $S\subseteq V$ with $|S|\ge2$, let $x\notin S$, and use the notation of Theorem~\ref{thm:one-point}.  Then
\begin{equation}\label{eq:one-point-equality}
        \rho_G(S\cup\{x\})=\rho_G(S)
        \quad\Longleftrightarrow\quad
        a=1
        \quad\Longleftrightarrow\quad
        u_{S\cup\{x\}}(x)=0.
\end{equation}
More precisely, the marginal increment is
\begin{equation}\label{eq:one-point-increment}
        \rho_G(S\cup\{x\})-\rho_G(S)
        =\rho_G(S)\rho_G(S\cup\{x\})\frac{(a-1)^2}{\eta}.
\end{equation}
If $u_S$ is the optimizer for $S$, then the same condition can be written as
\begin{equation}\label{eq:potential-equality}
        \sum_{s\in S}u_S(s)R_{xs}=\rho_G(S).
\end{equation}
\end{corollary}

\begin{proof}
The equivalence of the first three statements follows directly from \eqref{eq:one-point-inverse} and \eqref{eq:new-coordinate}.  Subtracting \eqref{eq:one-point-inverse} from $1/\rho_G(S)$ and clearing denominators gives \eqref{eq:one-point-increment}.  Finally, by \eqref{eq:u-star-inverse},
\[
        u_S=\rho_G(S)A^{-1}\one,
\]
so
\[
        \sum_{s\in S}u_S(s)R_{xs}=b^{\mathsf T}u_S=\rho_G(S)b^{\mathsf T}A^{-1}\one=\rho_G(S)a.
\]
Thus $a=1$ is equivalent to \eqref{eq:potential-equality}.
\end{proof}

\begin{remark}
When $S=\{s\}$, the matrix $R[S]$ is not invertible, so Theorem~\ref{thm:one-point} does not apply.  In this case the update is simply
\[
        \rho_G(\{s,x\})=\frac12R_{sx},
        \qquad \rho_G(\{s\})=0.
\]
\end{remark}

\section{Examples and non-convexity phenomena}\label{sec:examples}

\begin{example}\label{ex:complete}
Let $G=K_n$ with unit conductances.  The effective resistance between distinct vertices is $2/n$, so for every $S$ with $|S|=k\ge2$,
\[
        R[S]=\frac2n(J_k-I_k),
\]
where $J_k$ is the all-ones matrix and $I_k$ is the identity matrix.  The eigenvalues are $2(k-1)/n$ on the span of $\one$ and $-2/n$ on $\one^\perp$.  Hence
\[
        \rho_G(S)=\frac{2(k-1)}{nk}.
\]
This is increasing in $k$, in accordance with Corollary~\ref{cor:monotonicity}.  The optimizer is the uniform vector $u_S=k^{-1}\one$.
\end{example}

\begin{example}\label{ex:star}
Let $G$ be the unit-resistance star with center $0$ and leaves $1,2,3$.  Since $G$ is a tree, effective resistance is path distance.  In this example write $\rho=\rho_G$.  The values in Table~\ref{tab:star-values} follow by direct computation.

\begin{table}[ht]
\centering
\caption{Values of $\rho_G(S)$ for the unit-resistance star with center $0$ and leaves $1,2,3$.}
\label{tab:star-values}
\begin{tabular}{ll}
\toprule
$S$ & $\rho_G(S)$ \\
\midrule
$\{0,1\}$ & $1/2$ \\
$\{1,2\}$ & $1$ \\
$\{0,1,2\}$ & $1$ \\
$\{1,2,3\}$ & $4/3$ \\
$\{0,1,2,3\}$ & $3/2$ \\
$\{1\}$ & $0$ \\
\bottomrule
\end{tabular}
\end{table}

Supermodularity would require
\[
        \rho(A)+\rho(B)\le \rho(A\cap B)+\rho(A\cup B).
\]
Taking $A=\{0,1\}$ and $B=\{1,2\}$ gives
\[
        \rho(A)+\rho(B)=\frac32>1=\rho(A\cap B)+\rho(A\cup B),
\]
so supermodularity fails.

Submodularity would require the reverse inequality.  Taking $A=\{0,1,2\}$ and $B=\{1,2,3\}$ gives
\[
        \rho(A)+\rho(B)=1+\frac43=\frac73
        <\frac52=1+\frac32=\rho(A\cap B)+\rho(A\cup B),
\]
so submodularity also fails.  Hence monotonicity of $\rho_G$ is not the shadow of a global submodular or supermodular law.
\end{example}

\begin{example}\label{ex:update-interpretation}
Let $S\subseteq V$ with $|S|\ge2$ and let $u_S$ be the optimizer.  Define the potential generated by $u_S$ at a vertex $x$ by
\[
        P_S(x)=\sum_{s\in S}u_S(s)R_{xs}.
\]
Then Corollary~\ref{cor:one-point-equality} says
\[
        \rho_G(S\cup\{x\})=\rho_G(S)
        \quad\Longleftrightarrow\quad
        P_S(x)=\rho_G(S).
\]
If $P_S(x)\neq\rho_G(S)$, the increment is positive and its exact size is given by \eqref{eq:one-point-rho}.  Thus the equality set for monotonicity is an effective-resistance analogue of an equipotential set.
\end{example}

\section{Further comments}\label{sec:comments}

The results above separate the determinant of a principal resistance submatrix into two conceptually different factors.  The cofactor sum is enumerative: it is controlled by the ratio of the $S$-rooted forest enumerator to the spanning tree enumerator.  The normalized factor is potential-theoretic: it is a resistance-radius optimization problem and is most naturally evaluated on the Kron-reduced network.

For trees, the Kron-reduced expression can be expanded further because each pairwise effective resistance is a sum along a unique path.  This unique-path structure is responsible for the rooted-forest and floating-component formulas obtained for tree distance matrices in \cite{RichmanShokriehWuPrincipal}.  For general graphs, cycles introduce the denominator $\tau(G)$ into resistance entries, and a single positive forest-counting expression for $\det R[S]$ is not supplied by the present method; such an expression appears less direct without passing to ratios of forest polynomials.  The formula in Theorem~\ref{thm:det-principal} is therefore best viewed as a general network factorization: rooted forests supply the cofactor term, while Kron reduction supplies the intrinsic radius term.

The local formula \eqref{eq:one-point-inverse} may be useful computationally.  Once $R[S]^{-1}\one$ and $R[S]^{-1}$ are available, the effect of adding a candidate vertex $x$ can be computed from the two scalars $a$ and $\eta$.  The support criterion in Theorem~\ref{thm:equality-monotonicity} also shows that strict monotonicity is governed by the sign pattern of the optimizer $u_S$, which is closely related to the resistance-curvature viewpoint developed in \cite{Devriendt2022Thesis,DevriendtLambiotte2022}.

\section*{CRediT authorship contribution statement}
Guangfu Wang: Conceptualization, Methodology, Formal analysis, Investigation, Writing -- original draft, Writing -- review and editing, Funding acquisition.

\section*{Declaration of competing interest}
The author declares no known competing financial interests or personal relationships that could have appeared to influence the work reported in this paper.

\section*{Data availability}
No data were generated or analyzed during this study.

\section*{Declaration of generative AI and AI-assisted technologies in the manuscript preparation process}
During the preparation of this work the author used OpenAI ChatGPT to assist with language editing, journal-formatting checks, and LaTeX preparation.  After using this tool, the author reviewed and edited the content as needed and takes full responsibility for the content of the article.

\section*{Acknowledgements and funding}
This work was supported by the Natural Science Foundation of Shandong Province (Grant No. ZR2024MA073).

\end{document}